\documentclass[a4paper,10pt]{scrartcl}
\usepackage{pgfplots}
\usepackage[utf8]{inputenc}
\usepackage{lmodern}
\usepackage[english]{babel}
\usepackage[unicode, pdftex,
            pdfauthor={Wodson Mendson},
            pdftitle={Arithmetic aspects of the Jouanolou foliation},
            pdfsubject={foliations, algebraic geometry},
            pdfkeywords={foliations on surfaces, p-divisor},
            pdfproducer={Latex with hyperref},
            pdfcreator={Overleaf}]{hyperref}
\usepackage{amsmath}
\usepackage{amsfonts}
\usepackage{amssymb}
\usepackage{amsthm}
\usepackage{MnSymbol}
\usepackage{titling}
\usepackage{enumerate}
\usepackage{graphicx}
\usepackage{tabularx}

\DeclareOldFontCommand{\sc}{\normalfont\scshape}{\@nomath\sc}
\newtheorem*{resumo}{Abstract}

\newtheorem{propp}{Proposition}

\newtheorem*{thmII}{Theorem}
\newtheorem*{thmGao}{Proposition (Shuhong Gao)}

\newtheorem{theorem}{Theorem}

\newtheorem{pro}{Problem}
\newtheorem{thm}{Theorem}[section]
\newtheorem{lemma}[thm]{Lemma}
\newtheorem{cor}[thm]{Corollary}
\newtheorem{prop}[thm]{Proposition}
\newtheorem{OBS}[thm]{Remark}
\newtheorem{ex}[thm]{Example}
\newtheorem{dfn}[thm]{Definition}

\DeclareMathOperator{\corpo}{k}

\DeclareMathOperator{\sing}{sing}

\DeclareMathOperator{\Aut}{Aut}

\DeclareMathOperator{\Div}{Div}

\DeclareMathOperator{\ord}{ord}

\DeclareMathOperator{\h}{H^{0}}

\DeclareMathOperator{\spm}{\textbf{Spm}}

\DeclareMathOperator{\Projdois}{\mathbb{P}_{k}^{2}}

\usepackage[all]{xy}

\usepackage{authblk}
\usepackage{tikz-cd}

\title{Arithmetic aspects of the Jouanolou foliation}
\author{Wodson  Mendson}

\AtEndDocument{\bigskip{\footnotesize%
  \textsc{IRMAR, Université de Rennes 1, Campus de Beaulieu, 35402 Rennes Cedex
France} \par
  \textit{E-mail address}: \texttt{oliveirawodson@gmail.com} \par
  \addvspace{\medskipamount}
}}

\makeatletter
\def\keywords{\xdef\@thefnmark{}\@footnotetext}
\makeatother

\begin{document}



\maketitle
     \footnotetext{This work was developed at Université de Rennes 1. W. Mendson acknowledges support of CAPES}

\begin{resumo} \normalfont We investigate the structure  of the $p$-divisor for the Jouanolou foliation where we show, under some conditions, that it can be irreducible or has a $p$-factor. We study the reduction modulo $p$ of foliations on the projective plane and its applications to the problems of holomorphic foliations. We give new proof, via reduction modulo $2$, of the fact that the Jouanolou foliation on the complex projective plane of odd degree, under some arithmetic conditions, has no algebraic solutions.
\end{resumo}

\setcounter{tocdepth}{1}
\tableofcontents

\section{Introduction}

On the work \cite{mendson2022foliations} the author introduces the notion of the $p$-divisor for foliations on a smooth algebraic surface defined over a field characteristic $p>0$. This $p$-divisor appears as the degeneration divisor of the curvature morphism: this morphism takes a local section of the tangent bundle of the foliation and send it to its $p$-power.  The notion of $p$-divisor has been shown to be relevant to undestand the nature of the algebraic curves that are invariant by the foliation with applications to the study of codimension one holomorphic foliations on projective spaces. For instance, the struture of the $p$-divisor for generic foliations on $\mathbb{P}_{\corpo}^1\times \mathbb{P}_{\corpo}^1$ was a important step in the construction of new irreducible components of the space of codimension one holomorphic foliations on projective spaces done in the work \cite{mendson2022codimension} (see \cite[Theorem B]{mendson2022foliations} and \cite[Theorem 11.7]{mendson2022codimension}).

A fundamental property is the following: any algebraic invariant curve should appear in the support of the $p$-divisor and conversely, any prime divisor on the support of the $p$-divisor with multiplicity not zero modulo $p$ should be an invariant curve (see \cite[Proposition 3.8]{mendson2022foliations}).

In the study of the $p$-divisor appears naturally the following questions:

\begin{pro}\label{problem11} What we can say about the structure of the $p$-divisor for generic foliations? Is it irreducible, reduced?
\end{pro}

\begin{pro}\label{problem22} How many algebraic solutions can a foliation on a smooth projective surface have?
\end{pro}

A simple instance of the Problem \ref{problem22} occurs when we consider $p$-closed foliations. In this case, there are infinitely many solutions and the $p$-divisor is the zero divisor. In the work \cite{mendson2022foliations}, the author
gives an answer to the Problem \ref{problem11} in the case of foliations on the projective plane and in the Hirzebruch surfaces. It is explored the structure of the $p$-divisor where it is shown, under some conditions, that it is reduced, generically. On the projective plane, this is the following theorem (see \cite[Theorem A]{mendson2022foliations}).

\begin{thmII} Let $\corpo$ be a field of characteristic $p>0$. A generic foliation on the projective
plane $\mathbb{P}_{\corpo}^2$ of degree $d\geq 1$ with $p\nmid d-1$ has reduced $p$-divisor.
\end{thmII}

In this work, we investigate the struture of the $p$-divisor for the \textbf{Jouanolou foliations} where we give results in the direction of the Problem \ref{problem11} and Problem \ref{problem22}. The \textbf{Jouanolou foliation} of degree $d\geq 2$, $\mathcal{J}_d$, on the
projective plane over a field $K$ is given by the planar vector field:
$$
v_d =  (xy^d-1)\partial_x-(x^d-y^{d+1})\partial_y.
$$

The interesting fact about $\mathcal{J}_d$ is that when $K = \mathbb{C}$ the foliation $\mathcal{J}_d$ has no algebraic solutions for every $d\in\mathbb{Z}_{>1}$. This was proved by Jouanolou in his monograph \cite{MR537038}. On the other hand, when $K$ has characteristic $p>0$ it follows from \cite[Corollary 3.8]{mendson2022foliations} that the Jouanolou foliation has an algebraic solution for every $d\in \mathbb{Z}_{>1}$ such that $p\nmid d+2$. The main result is the following theorem.

\begin{theorem}\label{TeoremaA} Let $\corpo$ be an algebraically closed field of characteristic $p>0$. Let $d$ be a positive integer such that
\begin{itemize}
    \item $p<d$ and $p\not\equiv 1\mod 3$;

    \item $d^2+d+1$ is prime.
\end{itemize}
Then the Jouanolou foliation $\mathcal{J}_d$ on $\mathbb{P}_{\corpo}^2$ has irreducible $p$-divisor \textbf{or} $\Delta_{\mathcal{J}_d} = C+pR$ with $(\mathcal{J}_d,C)$ special and $R = \sum_{i}\alpha_i P_i$ with $P_i$ not $\mathcal{J}_d$-invariant for every $i$.
\end{theorem}

The notion of special pair $(\mathcal{J}_d,C)$ is introduced in the Section \ref{automorfismoIIII} (see Definition \ref{parspecial}). Under some arithmetic conditions on the characteristic of $\corpo$ we prove that the $p$-divisor of the Jouanolou foliation of \textbf{degree two} is always irreducible.

\begin{theorem} \label{TeoremaB} Let $p>2$ be a prime number such that $7\nmid p+4$ and such that $p\not\equiv 1\mod 3$. Then, the Jouanolou foliation of degree two, $\mathcal{J}_2$, defined over a field of characteristic $p$ has irreducible $p$-divisor.
\end{theorem}

As a consequence of the Theorem \ref{TeoremaA} and Theorem \ref{TeoremaB} we give a answer to the Problem \ref{problem22} to the Jouanolou foliations (see Corollary \ref{respostaproblemII}).

\begin{propp} In the conditions of Theorem \ref{TeoremaA}, the Jouanolou foliation of degree $d\geq 2$ in characteristic $p>0$ has an unique invariant curve and that curve has degree $pl+d+2$, with $0<l\leq d-1$.
\end{propp}

The proof of the Theorem \ref{TeoremaA} and Theorem \ref{TeoremaB} envolves some study of the influence of the automorphism group of a foliation on the structure of their $p$-divisor. This is explored in the Section \ref{automorfismoIIII}.

In the last part of the article, we present a simple proof of the fact that the Jouanolou foliation of odd degree, under some arithmetic conditions, has no algebraic solutions. This implies, in particular, the following result.

\begin{theorem}\label{TeoremaC} A very generic foliation of degree odd $d\geq 2$, with $d\not\equiv 1\mod 3$ on the projetive plane has no algebraic solutions.
\end{theorem}

Recall that the term \textbf{very generic} means that there exists a
countable union of closed sets $S$ in the space of holomorphic foliations of degree $d$ such that any foliations lying outside $S$ has no algebraic solutions. 

The Theorem \ref{TeoremaC} is well-known (see \cite{MR537038} and \cite[Theorem 1]{MR2156709}). The new content in this work is the method of the proof which uses only reduction to characteristic two and some irreducibility test for polynomials in two variables due to S.Gao (see \cite[Corollary 4.12]{MR1816701}).

\subsection{Organization of the paper} In Section \ref{notacoes} we fix the notations that will be used in the paper. In Section \ref{pdivisoremP2} we recall the definition of foliation on the projective plane, the notion of the $p$-divisor, automorphisms of foliation and we discuss some topics around reduction modulo $p$. In Section \ref{pdivisorP1P1}, we discuss some relations between foliations on the complex projective plane that has no algebraic solutions and foliations defined over a field of characteristic $p>0$ that has irreducible $p$-divisor. In Section \ref{automorfismoIIII}, we explore the automorphism group of foliations and its influence on the struture of the $p$-divisor where we present proofs of Theorem \ref{TeoremaA} and Theorem \ref{TeoremaB}. In the last section, we investigate the Jouanolou foliations in characteristic two and we present the proof of Theorem \ref{TeoremaC}.

\section{Notation} \label{notacoes}

\begin{itemize}

    \item $\mathcal{J}_d = $ the Jouanolou foliation;

    \item $\Delta_{\mathcal{F}}$ = the $p$-divisor associated to a foliation $\mathcal{F}$;

    \item $\spm(R) = $ the collection of maximal ideals of a domain $R$;

    \item $\mathbb{F}_{\mathfrak{q}} = $ the residue fied of $\mathfrak{q}\in \spm(R)$;
    
    \item $\mathbb{Z}[\mathcal{F}]=$ the $\mathbb{Z}$-algebra obtained by adjunction of all coefficients which appear in $\mathcal{F}$;

    \item $\mathcal{F}_\mathfrak{p} = $ the reduction modulo $\mathfrak{p}$ of $\mathcal{F}$;

    \item $\corpo[x,y,z]_e = $ $\corpo$-vector space of homogeneous polynomials of degree $e$ in $x,y,z$;

    \item $m_{Q}(C) =$ algebraic multiplicity of the curve $C \subset X$ at $Q \in X$;

    \item $O(2)=$ terms of the order at least $2$;

\end{itemize}

\section{Foliations on the projetive plane}\label{pdivisoremP2}

Let $d\in\mathbb{Z}_{>0}$ and $\corpo$ be an algebraically closed field. A foliation $\mathcal{F}$ on the projective plane $\mathbb{P}_{\corpo}^2$ is given, modulo $\corpo^{*}$, by a non-zero global section $\omega \in \h(\mathbb{P}_{\corpo}^2,\Omega_{\mathbb{P}_{\corpo}^2}^1(d+2))$ with finite singular locus. The degree of $\mathcal{F}$ has a geometric content: it is the number of tangencies between $\omega$ and a generic line $l \cong \mathbb{P}_{\corpo}^1$ in $\mathbb{P}_{\corpo}^2$. Using the Euler exact sequence, we can see $\omega$ as a homogeneous $1$-form on $\mathbb{A}_{\corpo}^3$: $\omega = Adx+Bdy+Cdz$ where
$$
Ax+By+Cz = 0 \quad \mbox{and} \quad\sing(\mathcal{F}) =  \mathcal{Z}(A,B,C)
  \mbox{ is finite}.
$$

\subsection{$p$-closed and non-$p$-closed foliations}

Let $\corpo$ be an algebraically closed field of characteristic $p>0$. Let $\mathcal{F}$ be a foliation of degree $d$ on $\mathbb{P}_{\corpo}^2$ and assume that $p$ does not divide $d+2$. Suppose that $\mathcal{F}$ is given by the projective $1$-form: $\omega = Adx+Bdy+Cdz$ and write
$$
d\omega = (d+2)(Ldy\wedge dz-Mdx\wedge dz+Ndx\wedge dy).
$$
Let $v_{\omega}$ be the homogeneous vector field on $\mathbb{A}_{\corpo}^3$ defined by:
$$
v_{\omega} = L\partial_x+M\partial_y+N\partial_z.
$$
Observe that by construction we have
$$
\mbox{\textbf{div}}(v_{\omega}) = \partial_xL+\partial_yM+\partial_zN = 0.
$$

We recall the following proposition form \cite{MR537038}.

\begin{prop} The map which associes $\omega$ to $v_{\omega}$ gives a bejection between the set $\h(\mathbb{P}_{\corpo}^2, \Omega_{\mathbb{P}_{\corpo}^2}^1(d+2))$ and the set of homogeneous vector fields of degree $d$ on $\mathbb{A}_{\corpo}^3$ with zero divergent.
\end{prop}
\begin{proof} See \cite[Proposition 1.1.4]{MR537038}.
\end{proof}

We say that $\mathcal{F}$ is \textbf{$p$-closed} if $i_{v_{\omega}^p}\omega = 0$, where $i_{v_{\omega}^p}$ denotes the contraction map. If $\mathcal{F}$ is not $p$-closed then we obtain a divisor $\Delta_{\mathcal{F}} = \{i_{v_{\omega}^p}\omega = 0\} \in \Div(\mathbb{P}_{\corpo}^2)$ of degree $p(d-1)+d+2$. This is called the \textbf{$p$-divisor} associated to $\mathcal{F}$. Recall that a reduced algebraic curve $C = \{F = 0\}$ is $\mathcal{F}$-\textbf{invariant} if there is $\Theta$, a $2$-homogeneous form of degree $d-1$, on $\mathbb{A}_{\corpo}^{3}$ such that
$$
dF \wedge \omega = F\Theta.
$$

A fundamental property of this notion is contained in the following proposition. 

\begin{prop} Let $\mathcal{F}$ be a non-$p$-closed foliation on $\mathbb{P}_{\corpo}^2$ and $C$ be an irreducible curve in $\mathbb{P}_{\corpo}^2$. Then:
    \begin{itemize}
          \item if $C$ is $\mathcal{F}$-invariant then $\ord_{\Delta_{\mathcal{F}}}(C)>0$;

          \item if $\ord_{\Delta_{\mathcal{F}}}(C)\not\equiv 0 \mod p$ then $C$ is $\mathcal{F}$-invariant.
    \end{itemize}
\end{prop}

\begin{proof} See \cite[Proposition 3.8]{mendson2022foliations}.
\end{proof}

Let $\mathcal{F}$ be a foliation on $\mathbb{P}_{\corpo}^2$ and $q \in \sing(\mathcal{F})$ be a singular point. Recall (see \cite[Definition 3.10]{mendson2022foliations}) that the point $q$ is called \textbf{$p$-reduced singularity} if there is a affine open subset $U\subset \mathbb{P}_{\corpo}^2$ which contains $q$ such that $\mathcal{F}|_{U}$ has linear part at $q$ of the form $ydx+\alpha xdy$ with $\alpha \not\in \mathbb{F}_p$. The foliation $\mathcal{F}$ is called \textbf{$p$-reduced foliation} if all singular point $q\in \sing(\mathcal{F})$ is $p$-reduced. If $\mathcal{F}$ has at lest a $p$-reduced singularity then $\mathcal{F}$ is not $p$-closed by \cite[Lemma 3.11]{mendson2022foliations}.

\begin{OBS} Let $X$ be a smooth surface defined over $\corpo$. A foliation $\mathcal{F}$ on $X$ can be defined as a system $\{(U_i,\omega_i,v_i)\}_i$ where $\{U_i\}$ is an open cover of $X$, $\omega_i \in \Omega_{X}^1(U_i)$ and $v_i\in T_{X}(U_i)$ satisfying compatibility conditions. The collection $\{\omega_i\}_i$ and $\{v_i\}_i$ define the \textbf{normal bundle}, $N_{\mathcal{F}}$, and the \textbf{tangent bundle}, $T_{\mathcal{F}}$, of the foliation. The foliation is not $p$-closed if $i_{v_i^p}\omega_i \neq 0$ for every $i$. In that case, the collection $\{i_{v_i^p}\omega_i\}_i$ defines a non-zero global section $s_{\mathcal{F}}$ of $(T_{\mathcal{F}}^{*})^{p}\otimes N_{\mathcal{F}}$ and the zero divisor of $s_{\mathcal{F}}$ is the \textbf{$p$-divisor} of $\mathcal{F}$ (see \cite[Section 3]{mendson2022foliations}).
\end{OBS}

\subsection{Automorphism of foliations}

Let $\corpo$ be an algebraically closed field and $\mathcal{F}$ a foliation on $\mathbb{P}_{\corpo}^{2}$ defined by a projective $1$-form $\omega \in \h(\mathbb{P}_{\corpo}^{2}, \Omega_{\mathbb{P}_{\corpo}^{2}}^{1}(d+2))$. An automorphism of $\mathcal{F}$ is an automorphism of the projective plane $\Phi\in \Aut(\mathbb{P}_{\corpo}^{2})$ which preserves $\mathcal{F}$. By preserves $\mathcal{F}$ we mean the following: if the foliation is given by the projective $1$-form $\omega \in \h(\mathbb{P}_{\corpo}^{2}, \Omega_{\mathbb{P}_{\corpo}^{2}}^{1}(d+2))$ then $\Phi^{*}\omega = \sigma(\Phi)\omega$ for some $\sigma(\Phi) \in \corpo^{*}$.

\begin{OBS} By the definition of $p$-divisor it follows that if $\mathcal{F}$ is not $p$-closed and $\Phi\in \Aut(\mathcal{F})$ then $\Phi^{*}\Delta_{\mathcal{F}} = \Delta_{\mathcal{F}}$.
\end{OBS}

The \textbf{Jouanolou foliation} defined on the projective plane $\mathbb{P}_{\corpo}^2$ of degree $d>1$ is given by the projective $1$-form:
$$
\mathcal{J}_d: \quad \omega = (x^{d}z-y^{d+1})dx+(xy^{d}-z^{d+1})dy+(z^{d}y-x^{d+1})dz.
$$
Suppose that the characteristic of $\corpo$ does not divide $d^2+d+1$. Let $\mu(\corpo)$ the group of the $(d^2+d+1)$th roots of the unity and let $\gamma$ be a generator of this group. Consider the automorphism $\Phi\colon \mathbb{P}_{\corpo}^2\longrightarrow \mathbb{P}_{\corpo}^2$ which associe $[x:y:z]\mapsto [\gamma^{d^2+1}x:\gamma y: z]$. Then $\Phi^{*}\omega = \gamma \omega$ so that $\Phi$ is an automorphism of $\mathcal{F}$. Note that $\Phi$ has order $d^2+d+1$. This automorphism will be used in the Section \ref{automorfismoIIII} to study the struture of the $p$-divisor of $\mathcal{J}_d$.

\subsection{Reduction modulo $p$}

Let $\mathcal{F}$ be a foliation of degree $d$ on the complex projective plane given by the projective $1$-form:
$$
\omega = Adx+Bdy+Cdz
$$ where $A,B,C \in \mathbb{C}[x_0,x_1,x_2]_{d+1}$. Let $\mathbb{Z}[\mathcal{F}]$ the $\mathbb{Z}$-algebra of finite type obtained by adjunction of all constant coefficients which appears on $A,B$ and $C$. Let $\mathfrak{p}\in \spm(\mathbb{Z}[\mathcal{F}])$ be a maximal ideal and let $p\in \mathbb{Z}_{>0}$ the characteristic of the residue field $\mathbb{F}_{\mathfrak{p}} =  \mathbb{Z}[\mathcal{F}]/\mathfrak{p}$. Note that $\mathbb{F}_{\mathfrak{p}}$ is a finite field by \cite[\href{https://stacks.math.columbia.edu/tag/00GB}{Tag 00GB}]{stacks-project} and, in particular, it has characteristic $p>0$. Denote by $\omega_{\mathfrak{p}} = \omega \otimes\mathbb{F}_{\mathfrak{p}}$ the section of $\Omega_{\mathbb{P}_{\overline{\mathbb{F}}_{\mathfrak{p}}}^2}^1(d+2)$ obtained by reduction modulo $\mathfrak{p}$ of the coefficients of $\omega$.

\begin{dfn} The foliation $\mathcal{F}_{\mathfrak{p}}$ defined by $\omega_{\mathfrak{p}}$ is called the \textbf{reduction modulo $\mathfrak{p}$} of $\mathcal{F}$.
\end{dfn}

Let $\mathcal{F}$ be a foliation on the complex projective plane and \textbf{P} be an abstract property. We say that \textbf{P} holds to $\mathcal{F}_{\mathfrak{p}}$ for \textbf{infinitely many primes $\mathfrak{p}$} if there is a dense subset $S\subset \spm(\mathbb{Z}[\mathcal{F}])$ such that for any $\mathfrak{m}\in S$ the property \textbf{P} holds to $\mathcal{F}_{\mathfrak{m}}$. Similary, we say that \textbf{P} holds to $\mathcal{F}_{\mathfrak{p}}$ for \textbf{almost all primes $\mathfrak{p}$} if there is a non-empty open subset $U \subset \spm(\mathbb{Z}[\mathcal{F}])$ such that for any $\mathfrak{m}\in U$ the property \textbf{P} holds to $\mathcal{F}_{\mathfrak{m}}$. 

The property \textbf{P} can be for example: $\mathcal{F}_{\mathfrak{p}}$ has an invariant curve of degree $e$, the foliation $\mathcal{F}_{\mathfrak{p}}$ is not $p$-closed or the $p$-divisor of $\mathcal{F}_{\mathfrak{p}}$ is irreducible.

\section{Algebraicity and the $p$-divisor}\label{pdivisorP1P1}

In this section, we investigate the relation between the class of foliations on the complex projective plane without algebraic solution and the class of foliations on the projective plane over characteristic $p>0$ that has irreducible $p$-divisor.  We begin the study with the following well-known basic lemma.

\begin{lemma}\label{basicao22} Let $X = \mathcal{Z}(f_1,\ldots,f_r)
\subset \mathbb{A}_{\mathbb{C}}^{m}$ be an affine variety and $A$ be the $\mathbb{Z}$-algebra of finite type obtained by adjunction of all coefficients which appers on $f_0,\ldots,f_r \in \mathbb{C}[x_1,\ldots,x_m]$. Then, $X(\mathbb{C})\neq \varnothing$ if and only if $X(\overline{\mathbb{F}}_{\mathfrak{q}})\neq \varnothing$ for infinitely many primes $\mathfrak{q}\in \spm(A)$
\end{lemma}
\begin{proof} This is well-known, but we will present a simple proof here of the non-trivial part (see \cite[Proposition 2.6]{MR1468476}). Suppose by contradiction that $X(\overline{\mathbb{F}}_{\mathfrak{q}})\neq \varnothing$ for infinitely many primes $\mathfrak{q}\in \spm(R)$ but $X(\mathbb{C})= \varnothing$. Let $S \subset \spm(A)$ the set of that maximal ideals. By the Nullstellensatz there is a relation:
$$
1  = \alpha_1f_1+\dots+\alpha_nf_n
$$
for some $\alpha_1,\ldots,\alpha_n \in \mathbb{C}[x_1,\ldots,x_m]$. Let $R$ the $A$-algebra of finite type obtained by inverting all constant coefficients of $\alpha_1,\ldots,\alpha_n$. Since, for almost all maximal ideal of $\mathbb{Z}$ is contraction of some maximal ideal of $R$ (see \cite[Proposition 4.1.5]{MR1790619}) we conclude that there are infinity of maximal ideals $\mathfrak{p} \in \spm(R)$ such that $\mathfrak{p}\cap A \in S$. So, for infinitely many ideals $\mathfrak{q}\in \spm(R)$ we have $1  \equiv \alpha_1f_1+\dots+\alpha_nf_n \mod \mathfrak{q}$, but this contradicts the fact $X(\overline{\mathbb{F}}_{\mathfrak{q}})\neq \varnothing$.
\end{proof}

\begin{prop}\label{globalzao} Let $\mathcal{F}$ be a foliation on the complex projective plane and let $d\in \mathbb{Z}_{>0}$. Then $\mathcal{F}$ has a $\mathcal{F}$-invariant curve of degree $d$ if and only if $\mathcal{F}_{\mathfrak{p}}$ has a $\mathcal{F}_{\mathfrak{p}}$-invariant curve of degree $d$ for infinitely many primes $\mathfrak{p}$.
\end{prop}

\begin{proof} Let $\corpo$ be an algebraically closed field and $\mathcal{G}$ a foliation on $\mathbb{P}_{\corpo}^2$ defined by a $1$-projective form $\Omega \in \h(\mathbb{P}_{\corpo}^2,\Omega^{1}_{\mathbb{P}_{\corpo}^2}(e+2))$. Note that the set of all $\mathcal{G}$-invariant curves of degree $d$ on $\mathbb{P}_{\corpo}^2$ is an projective variety defined over $\corpo$. Indeed, consider the following algebraic set:
$$
\mathcal{S}_d(\mathcal{G}) = \{([F],[\Theta]) \in \mathbb{P}(\mathcal{O}_{\mathbb{P}^{2}}(d))\times \mathbb{P}(V_{e-1})\mid \Omega\wedge dF-F\Theta = 0\}
$$
where $\mathbb{P}(\mathcal{O}_{\mathbb{P}^{2}}(d))$ denotes the projective space which parametrizes all homogeneous polynomial of degree $d$ and $\mathbb{P}(V_{e-1})$ parametrizes all homogeneous $3$-forms on $\mathbb{A}_{\mathbb{C}}^{3}$ of degree $e-1$. Then the set of all $\mathcal{F}$-invariant curves of degree $d$ is the set $S_d(\mathcal{G}) :=\pi_1(\mathcal{S}_d(\mathcal{G}))$ where $\pi_1\colon \mathbb{P}(\mathcal{O}_{\mathbb{P}^{2}}(d))\times \mathbb{P}(V_{e-1}) \longrightarrow \mathbb{P}(\mathcal{O}_{\mathbb{P}^{2}}(d))$ is the projection to the first factor. Since the projection is a proper map, we conclude that $S_d(\mathcal{G})$ is a projective variety.

By Lemma \ref{basicao22}, we conclude $S_d(\mathcal{F})(\mathbb{C})\neq \varnothing$ if and only if $S_d(\mathcal{F}_{\mathfrak{q}})(\overline{\mathbb{F}}_{\mathfrak{q}})\neq \varnothing$ for infinitely many primes $\mathfrak{q}$ and this proves the lemma.
\end{proof}

\begin{prop}\label{local_global}  Let $\mathcal{F}$ be a foliation on the complex projective plane of degree $d>1$. Then if the $p$-divisor of the foliation $\mathcal{F}_\mathfrak{p}$ is irreducible for infinitely many primes $p$ then $\mathcal{F}$ has no algebraic solutions.
\end{prop}

\begin{proof} Suppose that $\mathcal{F}_\mathfrak{q}$ has irreducible $q$-divisor for infinitely many primes $\mathfrak{q}$. Assume by contradiction that there is $C$ an irreducible curve that is invariant by $\mathcal{F}$. Since the invariance is preseved by reduction modulo $\mathfrak{q}$, we have that $C_\mathfrak{q}$ is a $\mathcal{F}_{\mathfrak{q}}$-invariant for infinitely many primes $\mathfrak{q}$. By \cite[Proposition 3.8]{mendson2022foliations} we know that $C_{\mathfrak{q}}\leq \Delta_{\mathcal{F}_{\mathfrak{q}}}$.  But, since $\Delta_{\mathcal{F}_{\mathfrak{q}}}$ is irreducible for infinitely many primes, we concude that $\Delta_{\mathcal{F}_{\mathfrak{q}}} = C_{\mathfrak{q}}$ for infinitely many primes $\mathfrak{q}$, a contradiction by degree comparison.
\end{proof}

\begin{OBS} There are foliations on the complex projective plane that has reduced $p$-divisor for almost all primes $\mathfrak{p}$. Indeed, consider the foliation $\mathcal{F}$ given in $\mathbb{A}^{2}_{\mathbb{C}}$ by the polynomial $1$-form:
$$
    \omega = ydx-xdy+a(x,y)dx+b(x,y)dy =ydx-xdy+\omega_2
$$
where $a(x,y),b(x,y)\in \mathbb{C}[x,y]$ are generic homogeneous polynomials of degree two. The foliation induced on $\mathbb{P}_{\mathbb{C}}^2$ by $\omega$ has four invariant lines:
the line at infinity $l_{\infty} = \{z=0\}$ and $l_1,l_2,l_3$ the factors of the polynomial $i_{R}\omega_2 = l_1l_2l_3$. Then for almost primes $\mathfrak{p}$ the $p$-divisor of $\mathcal{F}_{\mathfrak{p}}$ has the form
$$
\Delta_{\mathcal{F}_{\mathfrak{p}}} = l_1+l_2+l_3+l_{\infty}+C_{\mathfrak{p}}
$$
where $C_{\mathfrak{p}}$ is irreducible of degree $p$ (see \cite[Proposition 4.3]{mendson2022foliations}).
\end{OBS}

The next proposition follows from results that are contained in the paper \cite{mendson2022codimension} but we present a direct proof here (compare with \cite[Proposition 5.5]{mendson2022codimension}).

\begin{prop}\label{pullll} Let $\corpo$ be an algebraically closed field of characteristic $p>0$ and $X$ and $Y$ smooth projective surfaces defined over $\corpo$. Consider $
  \Phi\colon X \longrightarrow Y $ be a finite dominant separable morphism with ramification divisor $D$ and suppose that the reduced part of $\Phi(D)$ is not $\mathcal{F}$-invariant. Let $\mathcal{F}$ be  a foliation on $Y$ with irreducible $p$-divisor and let $\mathcal{G} = \Phi^{*}\mathcal{F}$ the induced folition on $X$. Then $\mathcal{G}$ is not $p$-closed with $p$-divisor given by $\Delta_{\mathcal{G}} = \Phi^{*}\Delta_{\mathcal{F}}+pD$, where $D$ denotes the the ramification divisor $D$.
\end{prop}

\begin{proof} Let $\omega \in N_{\mathcal{F}}^{*}$ be a local section defining $\mathcal{F}$ in an open set $U$. Since the reduced part $\Phi(D)$ is not $\mathcal{F}$-invariant we have $\Phi^*\omega$ a local section of $N_{\mathcal{G}}^{*}$ on $\Phi^{-1}(U)$, so that $\Phi^*N_{\mathcal{F}} = N_{\mathcal{G}}$. Now, the formula $K_{X} = \Phi^{*}K_{Y}+D$ and the adjunction formula implies:
$
      K_{\mathcal{G}} + N_{\mathcal{G}}^{*} =  K_{X} = \Phi^*K_{\mathcal{F}}+\Phi^{*}N_{\mathcal{F}}^{*}+D$ so that $K_{\mathcal{G}} = \Phi^{*}K_{\mathcal{F}}+D.$
So, $$\Delta_{\mathcal{G}} = pK_{\mathcal{G}}+N_{\mathcal{G}} = p(\Phi^{*}K_{\mathcal{F}}+D)+\Phi^{*}N_{\mathcal{F}} = \Phi^{*}\Delta_{\mathcal{F}}+pD$$ and this proves the result.
\end{proof}

\begin{ex} As an explicit example we consider the Jouanolou foliation $\mathcal{F}$ of degree $d=2$ over a field $\corpo$ of characteristic $p=3$. The foliation $\mathcal{F}$ is defined by the projective $1$-form:
$$
  \Omega = (x^2z-y^3)dx+(xy^2-z^3)dy+(z^2y-x^{3})dz.
$$

The foliation $\mathcal{F}$ is not $3$-closed with irreducible $3$-divisor given by
$$
  \Delta_{\mathcal{F}} = \{x^7+y^7+z^7+2x^{4}yz^{2}+2x^{2}y^{4}z+2xy^{2}z^{4} = 0\} \in \Div(\mathbb{P}_{\corpo}^2).
$$
Now consider the finite map $\Phi\colon \mathbb{P}_{\corpo}^2\longrightarrow \mathbb{P}_{\corpo}^2$
which associe $[x_0:x_1:x_2] \mapsto [x_0^2:x_1^2:x_2^2]$ and consider the induced foliation $\mathcal{G} = \Phi^*\mathcal{F}$. Then $\mathcal{G}$ is defined by the projective $1$-form
$$
\Phi^{*}\Omega = 2x(x^4z^2-y^6)dx+2y(x^2y^4-z^6)dy+2z(z^4y^2-x^6)dz
$$
with $3$-divisor given by
$$
\Delta_{\mathcal{G}} = \{x^{14}+y^{14}+z^{14}+2x^8y^2z^4+2x^4y^8z^2+2x^2y^4z^8=0\}+3\{x=0\}+3\{y=0\}+3\{z=0\}
$$
$$
 = \Phi^*\Delta_{\mathcal{F}}+3D.
$$
\end{ex}

\begin{cor} \label{casoruim} There are foliations on the projective plane $\mathbb{P}_{\mathbb{C}}^2$ that have no algebraic solutions and are non-$p$-closed for almost all primes $\mathfrak{p}$ with $p$-divisor having a $p$-power.
\end{cor}

\begin{proof} Let $\Phi\colon \mathbb{P}_{\mathbb{C}}^{2}\longrightarrow \mathbb{P}_{\mathbb{C}}^{2}$ be a finite morphism and $\mathcal{F}$ be a foliation on $\mathbb{P}_{\mathbb{C}}^{2}$ that has no algebraic solutions. Assume that foliation has at least a singular point with eigenvalue $\alpha \not\in \overline{\mathbb{Q}}$. Then $\mathcal{G} = \Phi^*\mathcal{F}$ is not $p$-closed for an almost all primes $p$ and does not have any algebraic solution. But, if $D$ is  the ramification divisor of $\Phi$ then by the Proposition \ref{pullll} it follows that the $p$-divisor, $\Delta_{\mathcal{G}_{\mathfrak{p}}}$, has $D\mod \mathfrak{p}$ as a $p$-power. In particular, it is not reduced for almost all primes $p$ and having a $p$-power.
\end{proof}

\begin{OBS} The Corollary \ref{casoruim}
shows that the converse of the Proposition \ref{local_global} does not holds in general.
\end{OBS}

Let $\mathcal{F}$ be a foliation on $\mathbb{P}_{\corpo}^2$ and $C$ be a $\mathcal{F}$-invariant curve. In positive characteristic, the singular locus of the invariant curve $C$ is not always contained in the singular locus of the foliation.

\begin{ex}\label{exemplao} Let $\mathcal{F}$ be the foliation on the complex projective plane given by the projective 1-form (see \cite[Theorem 4]{MR4389502}):
$$
 \Omega = -(xyz+y^{3}+z^{3})dx+(x^{2}z+y^{2}x-yz^{2})dy+z(xz+y^{2})dz.
$$
Consider the foliation $\mathcal{G} = \mathcal{F}\otimes 5\mathbb{Z}$ the projective plane over $\overline{\mathbb{F}}_5$ obtained by reduction modulo $5\mathbb{Z}$ of the coefficients of $\Omega$. Then the $5$-divisor of $\mathcal{G}$ is irreducible given by:
$$
\Delta_{\mathcal{G}} = \{-2x^3z^6-2x^2y^2z^5+xy^7z+2xyz^7+y^9+y^6z^3+y^3z^6-z^9=0\} \in \Div(\mathbb{P}_{\overline{\mathbb{F}}_5}^2).
$$
The foliation $\mathcal{G}$ has a unique singular point given by $[0:1:0]$. On the other hand, the curve $\Delta_{\mathcal{G}}$ has $[1:2:1]$ as singular point so that $\sing(\Delta_{\mathcal{G}})\nsubseteq \sing(\mathcal{G})$. Note that $\deg(\Delta_{\mathcal{G}}) = 9 \equiv 4 \mod 5$.
\end{ex}

We can say some information about the structure of the points in $\sing(C)-\sing(\mathcal{F})$. This is the content of the following proposition.

\begin{prop}\label{singularbao} Let $\corpo$ be an algebraically closed field of characteristic $p\geq2$ and $\mathcal{F}$ be a non $p$-closed foliation on $\mathbb{P}_{\corpo}^2$. Suppose that there is an irreducible curve, $C$, that is $\mathcal{F}$-invariant with degree $\deg(C)<p$. Then, $\sing(C)\subset \sing(\mathcal{F})$ or every $q\in \sing(C)-\sing(\mathcal{F}) $ is a hypercusp.
\end{prop}

\begin{proof} Suppose that there is a point $q\in \sing(C)-\sing(\mathcal{F})$. By the \cite[Corollary 2.4]{MR3687427} we know that there is a formal coordinate system $x,y$ on the completion of the ring $\mathcal{O}_{\mathbb{P}^2,q}$ such that the foliation is given by the vector field:
$$
v = \partial_x+x^{p-1}g(x,y)\partial_y
$$
with $g(x,y)\in \corpo[[x^p,y]]$. The curve in this coordinates system has form $f = f_e+f_{e+1}+\cdots\in \corpo[[x,y]]$ where $f_i$ is a homogeneous polynomial of degree $i$ in $x,y$. We need to show $f_e = l^e$ for some $l$, polynomial of degree $1$. Indeed, by definition of invariance we have $v(f) = fg$ for some $g\in \corpo[[x,y]]$  and considering the smallest term we obtain $\partial_x(f_e) =0$. Since, $e<\deg(C)<p$ we conclude that $f_e = y^e$.
\end{proof}

The difficult of the author in construct examples with $\sing(C)\not\subset \sing(\mathcal{F})$ and $\deg(C)<p$ leads to the following question.

\begin{pro} \label{invarianciasingular} Let $\corpo$ be an algebraically closed field of characteristic $p\geq 2$ and $\mathcal{F}$ be a non $p$-closed foliation on $\mathbb{P}_{\corpo}^2$. Suppose that there is an irreducible curve, $C$, that is $\mathcal{F}$-invariant with degree $\deg(C)<p$. Is it true that $\sing(C)\subset \sing(\mathcal{F})$?
\end{pro}

\begin{OBS} The Example \ref{exemplao} shows that the analogue of the Problem \ref{invarianciasingular} with the condition $\deg(C)\neq 0\mod p$ is not true.
\end{OBS}

\begin{prop} Assume that the answer for the Problem \ref{invarianciasingular} is YES. Let $\mathcal{F}$ be a holomorphic foliation on $\mathbb{P}_{\mathbb{C}}^2$ of degree two and suppose that it is non-degenerate with singularities having all eigenvalues not in $\mathbb{Q}$. Then, $\mathcal{F}$ has no algebraic solution if and only if $\mathcal{F}_{\mathfrak{p}}$ has irreducible $p$-divisor for infinitely many primes $\mathfrak{p}$.
\end{prop}

\begin{proof} If the $p$-divisor is irreducible for infinitely many primes $\mathfrak{p}$ then the result follows from Proposition \ref{local_global}. So, suppose that $\mathcal{F}$ does not have algebraic solutions. Since $\mathcal{F}$ is non-degenerate with all eigenvalues not a element of $\mathbb{Q}$ we conclude that $\mathcal{F}_{\mathfrak{p}}$ is not $p$-closed for an infinitely many primes $\mathfrak{p}$. Suppose by contradiction that $\Delta_{\mathcal{F}_\mathfrak{p}}$ is not irreducible for infinitely many primes $\mathfrak{p}$. Then for each that prime $\mathfrak{p}$ there is $C_{\mathfrak{p}}$ an irreducible factor of $\Delta_{\mathcal{F}_\mathfrak{p}}$ that is $\mathcal{F}_\mathfrak{p}$-invariant with $\deg(C_{\mathfrak{p}})<\deg(\Delta_{\mathcal{F}_\mathfrak{p}})$. There are two possibilities:
  \begin{itemize}
      \item $\deg(C_\mathfrak{p})\leq 4$ for infinitely many primes $\mathfrak{p}$, or

      \item $\deg(C_\mathfrak{p})>4$ for infinitely many primes $\mathfrak{p}$: in that case, since $\deg(\Delta_{\mathcal{F}_\mathfrak{p}}) = p+4$ there is a irreducible $\mathcal{F}_\mathfrak{p}$-invariant curve $D_{\mathfrak{p}}$ of degree less than $p$.
  \end{itemize}
In any case we conclude that for infinitely many primes $\mathfrak{p}$ there is a $\mathcal{F}_\mathfrak{p}$-invariant curve $D_{\mathfrak{p}}$ of degree $\deg(D_{\mathfrak{p}})<p$. Since we are assuming the Problem \ref{invarianciasingular} we conclude that for an infinitely of many primes $\mathfrak{p}$ we have $\sing(D_{\mathfrak{p}}) \subset \sing(\mathcal{F}_\mathfrak{p})$. In particular, since each eigenvalue of $\mathcal{F}$ is not a rational number, we have that $\alpha \mod \mathfrak{p}$ is not an element of $\mathbb{F}_p$ for infinitely many primes $\mathfrak{p}$. In particular, the singularities of the curve $D_{\mathfrak{p}}$ are nodes. By the \cite[Theorem 13]{MR1892310} we conclude that $\deg(D_{\mathfrak{p}})\leq 4$. So, for infinitely many primes $\mathfrak{p}$ the foliation $\mathcal{F}_\mathfrak{p}$ admits an invariant curve of degree $\leq 4$. By Proposition \ref{globalzao} it follows that $\mathcal{F}$ has an invariant curve of degree $\leq 4$, a contradiction.
\end{proof}

\section{Automorphism and the $p$-divisor}\label{automorfismoIIII}

In this section, motivated by the work \cite{MR2156709} we introduce the notion of special pair and study its main properties. We investigate the influence of the automorphism group of foliations on the degree of invariant curves and we prove Theorem \ref{TeoremaA} and Theorem \ref{TeoremaB}.

\begin{prop}\label{autbao} Let $\mathcal{F}$ be a foliation of degree $d\geq 2$ non-$p$-closed on $\Projdois$ over a field of characteristic $p>0$. Suppose that $p<d$ and that there exist $\Phi$ an automorphism of $\mathcal{F}$ of order $e$. Let $C$ be an irreducible curve and invariant by $\mathcal{F}$ such that $\Phi^{l} C \neq \Phi^{h}C$ if $l\neq h$ for every $l,h\in \{0,\ldots,e-1\}$. If $e \geq d^2$ then $\deg(C) = 1$.
\end{prop}

\begin{proof} Let $\psi = \Phi^{*}$ the isomorphism induced by $\Phi$ on the space of curves on $\mathbb{P}_{\corpo}^2$ of degree $\deg(C)$. This action is given via the pull-back by $\Phi$ of the polynomial equation defining the curve. Since $\Phi \in \Aut(\mathcal{F})$ and $ \psi^{h}C \neq \psi^{l}C$ for $h\neq l\in \{0,\ldots,e-1\}$ we have by \cite[Proposition 3.8]{mendson2022foliations} that
$$
\Delta_{\mathcal{F}}\geq C+\psi(C)+\cdots+\psi^{e-1}(C)  = \sum_{j=0}^{e-1}\psi^{j}C.
$$
In particular, $\deg(\Delta_{\mathcal{F}})\geq \sum_{j=0}^{e-1}\deg(\psi^{j}C)$ so that $\deg(\Delta_{\mathcal{F}}) \geq e\deg(C)$. Since $\deg(\Delta_{\mathcal{F}}) = p(d-1)+d+2$, $e\geq d^2$ and $p<d$ we get $d^2+2 = d(d-1)+d+2 > p(d-1)+d+2\geq d^2\deg(C)$. In particular, $2>d^2(\deg(C)-1)$ so that $\deg(C) = 1$. 
\end{proof}

We recall the following lemma (see \cite[Lemma 1]{MR2156709}).

\begin{lemma} Let $C$ be a $\mathcal{F}$-invariant curve. If $\corpo$ has characteristic $p>0$ suppose that $p$ does not divide $d+2$. Then there is an unique form $\beta_{\mathcal{F},C}$ such that
$$
  dF\wedge \omega = F\left(\left(\frac{\deg(C)}{d+2}\right)d\omega+i_{R}\beta_{\mathcal{F},C}\right)
$$
where $i_{R}$ denotes the contraction of $\beta_{\mathcal{F},C}$ with the radial vector field $R = x\partial_{x}+y\partial_y+z\partial_z$
\end{lemma}

\begin{proof} We recall the proof. The sequence $(x,y,z)$ of the polynomial ring $R = \corpo[x,y,z]$ is a regular sequence. From this, we conclude by \cite[\href{https://stacks.math.columbia.edu/tag/062F}{Tag 062F}]{stacks-project} that the Koszul complex induced by the radial vector field
$$
        \begin{tikzcd}
            0 \arrow[]{r}{} & \bigwedge^{3}R^{3} \arrow[]{r}{i_{R}} & \bigwedge^{2}R^{3} \arrow[]{r}{i_{R}} & R^{3} \arrow[]{r}{i_{R}} & R \arrow[]{r}{}&  0
        \end{tikzcd}
    $$
is exact. Now, consider the element $\sigma := dF\wedge \omega - (\frac{\deg(C)}{d+2})Fd\omega \in \bigwedge^{2}R^{3}$ and note that $i_{R}\sigma = 0$, so that there is an unique $\beta \in \bigwedge^{3}R^{3}\cong R$ such that $i_{R}\beta = \sigma$.
\end{proof}

\begin{dfn}\label{parspecial} Let $\mathcal{F}$ be a foliation on $\Projdois$ and $C$ be a $\mathcal{F}$-invariant curve. We say that the pair $(\mathcal{F}, C)$ is \textbf{special} if $\beta_{\mathcal{F},C} = 0$.
\end{dfn}

The degree of the pair $(\mathcal{F},C)$ is denoted by $(d,e)$ where $d$ is the degree of $\mathcal{F}$ and $e$ is the degree of $C$.

\begin{ex} If $\corpo$ is a field of characteristic $p>0$ and $\mathcal{F}$ is a non-$p$-closed foliation on $\mathbb{P}_{\corpo}^{2}$ with reduced $p$-divisor then $(\mathcal{F},\Delta_{\mathcal{F}})$ is special. This follows from \cite[Theorem 6.2]{MR2324555}
\end{ex}

\begin{prop}\label{speciallll} Let $(\mathcal{F},C)$ be a special pair of degree $(d,e)$ and suppose that $\mathcal{F}$ admits only $p$-reduced singularities. Then $\sing(\mathcal{F}) \subset C$ and $m_q(C) = 2$ for all $q\in \sing(\mathcal{F})$.
\end{prop}
\begin{proof} Let $q\in \sing(\mathcal{F})$ and suppose by contradiction that $q \not \in C$.  By a change of coordinates we can assume that $q = [0:0:1]$. Let $\Omega$ be a projective $1$-form defining $\mathcal{F}$ and $F\in \corpo[x,y,z]_e$ a homogeneous polynomial defining $C$. Since $(\mathcal{F},C)$ is special there is a relation:
$(d+2)\cdot dF\wedge \Omega = e\cdot Fd\Omega$. Since we are assuming that $q\not \in C$ in the affine open subset $D_{+}(z) \cong \mathbb{A}_{\corpo}^2$ the have the equation $(d+2)\cdot df \wedge \omega = e\cdot fd\omega$ with
$$
\omega = \omega_1+\omega_2+\cdots \qquad, \qquad f = f_0+f_1+f_2+\cdots
$$
where $\omega_i = a_i(x,y)dx+b(x,y)dy$ with $a_i,b_i$ homogeneous polynomials of degree $i$, $f_h$ homogeneous polynomial of degree $h$, $\omega_1 = ydx+\alpha xdy$, $f_0\in \corpo^*$ and $\alpha \in \corpo-\mathbb{F}_p$. In particular, $\sigma := df/f \in \Omega_{\mathbb{P}_{\mathbb{C}}^2,q}^{1}$, i.e., it is regular at point $q$. So, if $\mathcal{M}_{\mathbb{P}_{\corpo}^2,q}$ is the maximal ideal of the point $q$ in $\mathbb{P}_{\corpo}^2$ we have
$$
0 = (d+2)\cdot \sigma\wedge \omega \mod \mathcal{M}_{\mathbb{P}_{\mathbb{C}}^2,q} = e\cdot f_0 d\omega_1  \mod \mathcal{M}_{\mathbb{P}_{\mathbb{C}}^2,q} = e\cdot f_0 (\alpha -1) dx\wedge dy \mod \mathcal{M}_{\mathbb{P}_{\mathbb{C}}^{2},q}
$$
so that $\alpha = 1$. But, this is a contradiction since we are assuming that $q$ is $p$-reduced. So, $q \in C$. Now, suppose that $f_1\neq 0$, that is, $q$ is a regular point of $C$ and write $f_1 = ux+vy$ where $u\in \corpo^*$ or $v\in\corpo^*$. Considering the smallest term in the equation $(d+2)\cdot df \wedge \omega = e\cdot fd\omega$ we get
$$
(d+2)\cdot df_1\wedge \omega_1 = e\cdot f_1d\omega_1
$$
and by expansion we get two equations:
$$
 u\alpha (d+2) = ue(\alpha-1) \qquad, \qquad -v(d+2) = ev(\alpha-1).
$$
Now, if $u\neq 0$ then we conclude that $1/\alpha =  (1-(d+2)/e)$. In particular, $\alpha \in \mathbb{F}_p$, a contradiction. If $v\neq 0$ then $\alpha = 1-(d+2)/e \in\mathbb{F}_p$, a contradiction. So, we conclude that $u = v =0$ so that $f_1 = 0$. In particular, $\sing(\mathcal{F}) \subset \sing(C)$. Now since $q$ is $p$-reduced by \cite[Lemma 3.11]{mendson2022foliations} we conclude that $m_q(\Delta_{\mathcal{F}}) = 2$ and since $C\leq \Delta_{\mathcal{F}}$ we get $m_q(C)=2$. This ends the proof of the proposition.
\end{proof}

\begin{cor} \label{grauspecial} Assume that $\corpo$ has characteristic $p>2$. Let $(\mathcal{F},C)$ be a special pair of degree $(d,e)$ and suppose that $\mathcal{F}$ admits at least a $p$-reduced singularity. Then $\deg(C) \equiv d+2\mod p$.
\end{cor}

\begin{proof} Let $q$ be a $p$-reduced singularity of $\mathcal{F}$. Without loss of generality, we can assume that $q = [0:0:1]$. In the affine open set $D_{+}(z)$ the foliation is given by a $1$-form $\omega = \omega_1+\omega_2+O(2)$ and by Proposition \ref{speciallll} the curve given by $f = f_2+O(3)$ with $f_2 = a_1x^2+a_2xy+a_3y^2$ with $a_i\neq 0$ for some $i$. Considering the smallest term in the equation $(d+2)\cdot df\wedge \omega = e\cdot fd\omega$ we obtain:
$$
(d+2)df_2\wedge \omega_1 = e\cdot f_2d\omega_1
$$
and expanding that equation we obtain the following equations over the field $\corpo$:
$$
2a_1\alpha (d+2) = a_1e(\alpha -1)\quad,\quad a_2e(\alpha-1) = (d+2)a_2(\alpha-1) \quad\mbox{and}\quad a_3e(\alpha -1)= -2(d+2)a_3.
$$
If $a_1\neq 0$ or $a_3\neq 0$ then we get a contradiction since $\alpha \not \in \mathbb{F}_p$. So, $a_1 = a_3 = 0$ and we conclude that $a_2\neq 0$. In particular, we obtain $d+2 = e$ identity over the field $\corpo$. This proves the corollary.
\end{proof}

\begin{OBS} The proof of the Corollary \ref{grauspecial} shows that all singularities of $C$ along $\sing(\mathcal{F})$ are nodal.
\end{OBS}

\begin{OBS}\label{retasJOU} It follows from Corollary \ref{grauspecial} that a foliation, $\mathcal{F}$, defined over a field of characteristic $p>0$ has a $p$-reduced singular point then $\mathcal{F}$ does not have any invariant lines.
\end{OBS}

In the next proposition, we use the fact that a $p$-reduced foliation of degree $d$ on the projective plane has $d^2+d+1$ distict singularities. This follows from the computations in \cite[Proposition 2.1]{MR3328860}.

\begin{lemma} \label{speciallll2} Assume that $\corpo$ has characteristic $p>2$. Let $\mathcal{F}$ be a non $p$-closed foliation of degree two on the projective plane $\mathbb{P}_{\corpo}^2$ and suppose that $\mathcal{F}$ is $p$-reduced. Let $C = \{F = 0\}$ be an irreducible $\mathcal{F}$-invariant curve with $\deg C < \deg \Delta_{\mathcal{F}}$. Then $(\mathcal{F},C)$ is not special.
\end{lemma}

\begin{proof} Suppose by contradiction that $(\mathcal{F},C)$ is special. The $p$-divisor has the form $\Delta_{\mathcal{F}} = C+R$ and by the Corollary \ref{grauspecial} we know that $\deg(C)\in \{4, 4+p\}$ and since $\deg(C)<\deg(\Delta_{\mathcal{F}})$ we conclude $\deg(C) = 4$. Proposition \ref{speciallll} ensures that $\sing(\mathcal{F})\subset \sing(C)$ and $m_{q}(C) = 2$ for all $q\in \sing(\mathcal{F})$. Now, by  \cite[Theorem 2, page 60]{MR0313252} we have
$$
  7 = \#\sing(\mathcal{F}) \leq \sum_{q\in \sing(C)} \frac{m_{q}(C)(m_{q}(C)-1)}{2} \leq \frac{(4-1)(4-2)}{2} = 3
$$
a contradiction. So, $(\mathcal{F},C)$ is not special.
\end{proof}

\begin{lemma}\label{Jounaopfechada} Let $p>2$ be a prime number such that $p\not \equiv 1 \mod 3$. Let $d\in \mathbb{Z}_{>1}$ be an integer and suppose that $p\nmid d(d+2)(d^2+d+1)$. Then the Jouanolou foliation of degree $d$, $\mathcal{J}_d$, over $\corpo$ is not $p$-closed.
\end{lemma}

\begin{proof} Note that the Jouanolou foliation over $\corpo$ is just the reduction modulo $p$ of the Jouanolou foliation defined on the complex projective plane. So, by the \cite[Lemma 3.11]{mendson2022foliations} it is sufficient to show that the reduction modulo $p$ of the eigenvalues of $\mathcal{J}_d$ is not an element of $\mathbb{F}_p$. By the properties of the Jouanolou foliation, we know that the automorphism $\Phi_d$ acts transitively on the set $\sing(\mathcal{J}_d)$ and this implies that all the singularities of $\mathcal{J}_d$ has the same eigenvalues $\alpha\in \mathbb{C}$. The Baum-Bott formula (see \cite[Theorem 3.1]{MR537038}) implies $(d+2)^2 = (d^2+d+1)(\alpha+\alpha^{-1}+2)$ and this implies that $\alpha$ and $\alpha^{-1}$ are the roots of the polynomial $p(x) = x^2+((d^2-2d-2)/(d^2+d+1))x+1\in \mathbb{C}[x]$.

The polynomial $p(x)$ has discriminant equal to $\delta = -3d^2(d+2)^2/(d^2+d+1)^2$. In particular, $\alpha \mod p$ is not an element of $\mathbb{F}_p$ if and only if $-3$ is not square modulo $p$. By the law of quadratic reciprocity, we conclude that $\mathcal{J}_d$ is not $p$-closed if and only if $p\not\equiv 1\mod 3$.
\end{proof}

\begin{thm} \label{JouJouJou} Let $p>2$ be a prime number such that $7\nmid p+4$ and such that $p\not\equiv 1 \mod 3$. Then, the Jouanolou foliation of degree two $\mathcal{J}_2$ defined over a field of characteristic $p$ has irreducible $p$-divisor.
\end{thm}

\begin{proof} The Jouanolou foliation on the $\mathbb{P}_{\corpo}^2$ is given by the projective $1$-form:
$$
\omega = (x^2z-y^3)dx+(xy^2-z^3)dy+(z^2y-x^3)dz.
$$
Let $\gamma \in \corpo$ be a primitive $7$th root of unity. Note that since  $p\neq 7$ the morphism
$$\Phi\colon \mathbb{P}_{\corpo}^2\longrightarrow \mathbb{P}_{\corpo}^2\qquad [x:y:z]\mapsto [\gamma^5 x:\gamma y:z]$$
is a automorphism of the foliation $\mathcal{J}_2$ with $\Phi^{*}\omega = \gamma \omega$.

The condition $p\not\equiv 1 \mod 3$ implies that the Jouanolou foliation is not $p$-closed by Lemma \ref{Jounaopfechada}. Suppose by contradiction that $\mathcal{J}_2$ has non-irreducible $p$-divisor. In particular, there is a irreducible $\mathcal{J}_2$-invariant curve $C = \{F = 0\}$ with
$\deg(C)<\deg(\Delta_{\mathcal{J}_2})$. Lemma \ref{speciallll2} ensures that the pair $(\mathcal{J}_2,C)$ is not special. In particular, there is a $3$-homogeneous form $\beta$ of degree one on $\mathbb{A}_{\corpo}^3$ such that
$$
dF\wedge \omega = F\Theta \quad \mbox{where} \quad \Theta = \frac{\deg(C)}{4} d\omega + i_{R}\beta.
$$
Now we will consider two cases:
\begin{itemize}
 \item \textbf{some power $\Phi^{l}$ fixes the curve $C$}:  In that case, we conclude that $(\Phi^l)^*\beta = \gamma^{l}\beta$. If $\beta = H(x,y,z)dx\wedge dy \wedge dz$ with $H(x,y,z) = ax+by+cz\in \corpo[x,y,z]_1$ we get the equations:
 $$
 \gamma^{10l} a =a \quad,\quad \gamma^{6l} b = b \quad,\quad \gamma^{5l} c = c
 $$
 so that $a = b = c =0$, a contradiction since $C$ is not special.

 \item \textbf{$\Phi^{l}$ does not fix the curve $C$ for any $l\in \{0,\ldots, 6\}$:} In particular, $(\Phi^{*})^{l}C\neq (\Phi^{*})^{k}C$ for every $k\neq l\in \{0,\ldots,6\}$: Since $\deg(\Delta_{\mathcal{J}_2}) = p+4$ and $\Phi$ has order $7$ we observe that $\Delta_{\mathcal{J}_2}$ does not have $p$-factor. So, there are two possibilities: There is a prime divisor $P$ in the support of $\Delta_{\mathcal{J}_2}$ that is fixed by some power of $\Phi$ or there is not that prime. In the first case, we use the previous argument in order to get a contradiction. In the last case, we conclude that $7\mid p + 4$, a contradiction.
\end{itemize}
This ends the proof of the theorem.
\end{proof}

\begin{cor} In the conditions of the Theorem \ref{JouJouJou}, a generic foliation of degree two on the projective plane over characteristic $p>0$ has irreducible $p$-divisor.
\end{cor}

\begin{cor} The Jouanolou foliation of degree two $\mathcal{J}_2$ on $\mathbb{P}_{\mathbb{C}}^2$ has no algebraic solutions.
\end{cor}

\begin{proof} By Theorem \ref{JouJouJou} we know that the reduction modulo $p$ of $\mathcal{J}_2$ has irreducible $p$-divisor for infinitely many primes $p$. Now, the result follows from Proposition \ref{local_global}.
\end{proof}

\begin{OBS} If $p=7$ the argument used in the corollary above does not works because in that case there is no primitive $7$th root. Using the computer algebra system \cite[Singular]{DGPS} we check that the $7$-divisor of the Jouanolou foliation $\mathcal{J}_2$ is given by
$$
\Delta_{\mathcal{J}_2} = \{-3x^{3}z+xy^{3}+2yz^{3} = 0\}+7\{2x+y-3y = 0\}.
$$
\end{OBS}

\begin{prop}\label{chavechave} Let $\mathcal{F}$ be a $p$-reduced foliation on $\Projdois$ and let $C$ be an irreducible $\mathcal{F}$-invariant curve. Suppose that $\mathcal{F}$ is not $p$-closed and that $\deg(C)<\deg(\Delta_{\mathcal{F}})$. Let $\Phi\in \Aut(\mathcal{F})-\{1\}$ be an automorphism of $\mathcal{F}$. If $(\mathcal{F},C)$ is special then there is a $\mathcal{F}$-invariant irreducible curve $D$ such that $(\mathcal{F},D)$ is not special \textbf{or} $\Phi$  fixes the curve $C$.
\end{prop}
\begin{proof} Suppose that $\Phi$ does not fix the curve $C$. Write $\Delta_{\mathcal{F}} -C = \sum_{i}\alpha_iP_i$ where $P_i$'s are prime divisors. By Corollary \ref{grauspecial} we know that $\deg(C) = d+2+pl$ for some $l\in \{0,\ldots,d-2\}$ so that $\sum_{i}\alpha\deg(P_i) = p(d-1-l)$. We have two possibilities:

\begin{itemize}
 \item \textbf{$\alpha_{i}\equiv 0 \mod p$ for all $i$:} In that case, if $\alpha_i = pl_i$ for some $l_i\in \mathbb{Z}_{\geq 0}$ then we conclude $\sum_{i}l_i\deg(P_i) = d-1-l < d$. In particular, $\deg(P_i)<d$. But this is a contradiction since $\Delta_{\mathcal{F}} = \Phi^{*}\Delta_{\mathcal{F}}  = \Phi^{*}C+\sum_{i}\alpha_{i}\Phi^{*}P_{i}$ with $C\neq \Phi^{*}C$ and $C<\Delta_{\mathcal{F}}$.
 
 \item \textbf{$\alpha_i \not\equiv 0 \mod p$ for some $i$:} In that case, $P_i$ is a $\mathcal{F}$-invariant divisor by \cite[Proposition 3.8]{mendson2022foliations}. Now, note that $(\mathcal{F},P_i)$ is not special. Indeed, assume by contradiction that $(\mathcal{F},P_i)$ is special. Then, by Proposition \ref{speciallll} we know that $\sing(\mathcal{F}) \subset \sing(P_i)$. In particular, for every $q \in\sing(\mathcal{F}) $ we have $2 = m_q(\Delta_{\mathcal{F}}) \geq m_q(C)+m_q(P_i)\ = 4$, a contradiction.
\end{itemize}
This ends the proof of the proposition.
\end{proof}

\begin{thm} \label{irredutibilidade} Let $\corpo$ be an algebraically closed field of characteristic $p>0$. Let $d>2$ be a positive integer such that
\begin{itemize}
    \item $p<d$ and $p\not\equiv 1\mod 3$;

    \item $d^2+d+1$ is prime.
\end{itemize}
Then the Jouanolou foliation $\mathcal{J}_d$ on $\mathbb{P}_{\corpo}^2$ has irreducible $p$-divisor \textbf{or} $\Delta_{\mathcal{J}_d} = C+pR$ with $(\mathcal{J}_d,C)$ special, $\deg(C) = pl+d+2$ with $l>0$ and $R = \sum_{i}\alpha_i P_i$ with $P_i$ not $\mathcal{J}_d$-invariant for every $i$.
\end{thm}

\begin{proof} The condition  $p\not\equiv 1\mod 3$ implies that $\mathcal{J}_d$ is not $p$-closed by Lemma \ref{Jounaopfechada}. Suppose that there is $C$ a $\mathcal{F}$-invariant curve such that $\deg(C)<\deg(\Delta_{\mathcal{J}_d})$.  Consider the pair $(\mathcal{J}_d,C)$ and the automorphism:
$\Phi_{d}\colon \mathbb{P}_{\corpo}^2\longrightarrow \mathbb{P}_{\corpo}^2\in\Aut(\mathcal{J}_d)$ which associes $[x:y:z]\mapsto [\gamma^{d^2+1} x:\gamma y:z].$
We have two cases:
    \begin{itemize}

          \item \textbf{$(\mathcal{J}_d,C)$ is not special:} in this case we consider two subcases: suppose first that $\Phi_d^{l}$ does not fix $C$ for every $l\in\{0,\ldots,d^2+d\}$. In this case, by Proposition \ref{autbao} we conclude that $\deg(C)=1$, a contradiction since $\mathcal{J}_d$ does not have invariant lines (see Remark \ref{retasJOU}). Consider the case where $\Phi_d^l$ fixes $C$ for some $l$ and consider to the identity:
              $$
                      dF\wedge \omega = F\Theta_C \quad \mbox{where} \quad \Theta_C = \frac{\deg(C)}{d+2} d\omega + i_{R}\beta.
              $$
          Applying the automorphism $\Phi_d^l$ on the equation above we conclude that
          $$
                (\Phi_d^l)^{*}\beta = \gamma^{l}\beta.
          $$
          Write $\beta = b(x,y,z)dx\wedge dy\wedge dz$ and $b(x,y,z) = \sum_{0\leq j<i\leq d-1 }b_{ij}x^jy^{i-j}z^{d-1-i}$. Then, since $\Phi_d^{l}$ associes $[x:y:z]\mapsto [\gamma^{l(d^2+1)}x:\gamma^l y:z]$ we conclude that $b(x,y,z) = \gamma^{l(d^2+1)}b(\gamma^{l(d^2+1)}x,\gamma^l y,z) = \gamma^{-ld}b(\gamma^{l(d^2+1)}x,\gamma^l y,z)$, so that if $b_{ij}\neq  0$ then
          $$(\gamma^{j(d^2+1)+i-j-d})^l = (\gamma^{-dj+i-j-d})^l = 1\Longrightarrow (\gamma^{dj-i+j+d})^l =1.$$
          Since $d^2+d+1$ is prime we conclude that $d^2+d+1\mid dj-i+j+d$. But, $dj-i+j+d<dj+j+d\leq d(d-1)+d-1+d = d^2+d-1 < d^2+d+1$, a contradiction.

          \item \textbf{$(\mathcal{J}_d,C)$ is special:} by Proposition \ref{chavechave} we know that there is a $\mathcal{J}_d$-invariant prime divisor $D$ such that $(\mathcal{J}_d,D)$ is not special \textbf{or} the automorphism $\Phi_d$ fixes the curve $C$. By the precedent case, it follows that for every $\mathcal{J}_d$-invariant prime divisor $D$ with $\deg(D)<\deg(\Delta_{\mathcal{J}_d})$ the pair $(\mathcal{J}_d,D)$ is special. So, it follows that $\Phi_d$ fixes $C$. Write $$\Delta = C+\sum_i \alpha_i P_i$$
          and observe that $P_i's$ are not $\mathcal{J}_d$-invariant for every $i$. Indeed, if $P_i$ is $\mathcal{J}_d$-invariant for some $i$, then the pair $(\mathcal{J}_d,P_i)$ is special. By Proposition \ref{speciallll} we conclude that $2 = m_q(\Delta_{\mathcal{J}_d})\geq m_q(C)+m_q(P_i) \geq 2+2$ for all $q\in \sing(\mathcal{F})$, a contradiction. In particular, by \cite[Proposition 3.8]{mendson2022foliations} we conclude that $\alpha_i \equiv 0 \mod p$ for every $i$ and the $p$-divisor has the form $\Delta_{\mathcal{J}_d} = C+pR$ with $(\mathcal{J}_d,C)$ special and $R = \sum_{i}\alpha_i P_i$ with $P_i$ not $\mathcal{J}_d$-invariant for every $i$.
    \end{itemize}
This ends the proof of the theorem.
\end{proof}

\begin{ex} Consider the Jouanolou foliation of degree $d\leq 100$ over characteristic $p = 5$. Using the computer algebra system \cite[Singular]{DGPS} we can check that the $5$-divisor is irreducible if $d\in \{2,14,24,54,59,69,89,99\}$. Otherwise, $\mathcal{J}_d$ is $5$-closed or $\Delta_{\mathcal{J}_d} = C+pR$ with structure given by the following table:

\begin{center}

 \begin{tabularx}{0.8\textwidth} {
  | >{\centering\arraybackslash}X
  | >{\centering\arraybackslash}X
  | >{\centering\arraybackslash}X
  | >{\centering\arraybackslash}X
  | >{\centering\arraybackslash}X|
  }
   \hline
  $d$  & $d^2+d+1$ & $\deg(C)$ &  $l$ &  $R$ \\\hline
  $6$ & $43$   &   $18$ &    $2$   & $\{xyz=0\}$\\\hline
  $12$ & $157$  &   $39$ &    $5$  & $\{xyz=0\}$\\\hline
  $17$ & $307$  &   $54$ &    $7$  & $\{xyz=0\}$ \\\hline
  $21$ & $463$  &   $63$  &    $8$  &  $\{xyz=0\}$\\\hline
  $27$ & $757$  &   $84$  &    $11$  &  $\{xyz=0\}$\\\hline
  $41$ & $1723$ & $123$ &   $16$   &  $\{xyz=0\}$\\\hline
  $57$ & $3307$ & $174$  &   $23$  &  $\{xyz=0\}$\\\hline
  $62$ & $3907$ & $189$  &   $25$  &  $\{xyz=0\}$\\\hline
  $66$ & $4423$ &  $198$ &   $26$  &  $\{xyz=0\}$\\\hline
  $71$ & $5113$ & $213$ &   $28$  &  $\{xyz=0\}$\\\hline
  $77$ & $6007$ & $234$  &   $31$  &  $\{xyz=0\}$\\\hline

  \end{tabularx}

 \end{center}
where $(\mathcal{J}_d,C)$ is a special pair of degree $(d,5l+d+2)$. Note that $p$ and $d$ satisfy the conditions of the Theorem \ref{irredutibilidade}.
\end{ex}

\begin{cor}\label{respostaproblemII} In the conditions of Theorem \ref{irredutibilidade} the Jouanolou foliation on the projective plane $\Projdois$ has an unique irreducible invariante curve of degree $pl+d+2$ with $0<l\leq d-1$.
\end{cor}

\begin{OBS} It is expected that for infinitely many of $d\in \mathbb{Z}_{>0}$
the integer $d^2+d+1$ is a prime number. This is a particular case of the Bunyakovsky conjecture which says that, under some conditions, given a polynomial $p(x) \in \mathbb{Z}[x]$ we have $p(d)$ prime for infinitely many $d\in \mathbb{Z}_{>0}$ (see \cite{438807}).
\end{OBS}

\section{Non-algebraicity via reduction modulo \texorpdfstring{$2$}{2}} \label{pdivisorSigma}

In this section, we prove a particular case of a well-known theorem of Jouanolou which says that a very generic foliation of degree $d>1$ on the projective plane does not have any algebraic solutions. In this section, we present a simple proof of this result in the case where $d$ is odd with $d\not \equiv 1 \mod 3$. The main topic in this proof is that we use reduction modulo two and some irreducibility criterion for polynomials in two variables via Newton polytopes due to S.Gao.

Let $\mathcal{S} = \{P_1,\ldots,P_n\}$ be a collection of points on $\mathbb{R}^2$. Recall that the \textbf{convex hull} of $\mathcal{S}$ denoted by $\langle\mathcal{S}\rangle = \langle P_1,P_2,\ldots,P_{n-1},P_n\rangle$ is the smallest convex subset of $\mathbb{R}^2$ which contains $S$. This is equivalent to say that $\langle\mathcal{S}\rangle $ is the intersection of all convex subsets of $\mathbb{R}^2$ which contains $S$. It is possible to check that
$$
\langle\mathcal{S}\rangle = \{t_{1}P_{1}+t_2P_2+\cdots +t_{n-1}P_{n-1}+t_nP_n\in \mathbb{R}^2\mid \sum_{i}t_i=0\}.
$$
Let $\corpo$ be a field and $f(x,y) = \sum_{i,j}a_{ij}x^iy^j \in \corpo[x,y]$ a polynomial. Recall that the \textbf{Newton polytope} associated to $f(x,y)$ is the convex hull os the set $\mathcal{S}(f) = \{(i,j)\in \mathbb{Z}^2\mid a_{ij}\neq 0\}$.

\begin{ex} Let $f(x,y) = x^{2}y+xy^{4}+y^{2}+xy^{2}\in \corpo[x,y]$. In this case, $\mathcal{S}(f) = \{(2,1),(1,4),(0,2),(1,2)\}$. The Newton polytope associated to $f$ is the triangle with vertices $A = (2,1)$, $B = (1,4)$ and $C = (1,2)$.

\centering
\includegraphics[width=5cm , height=3cm]{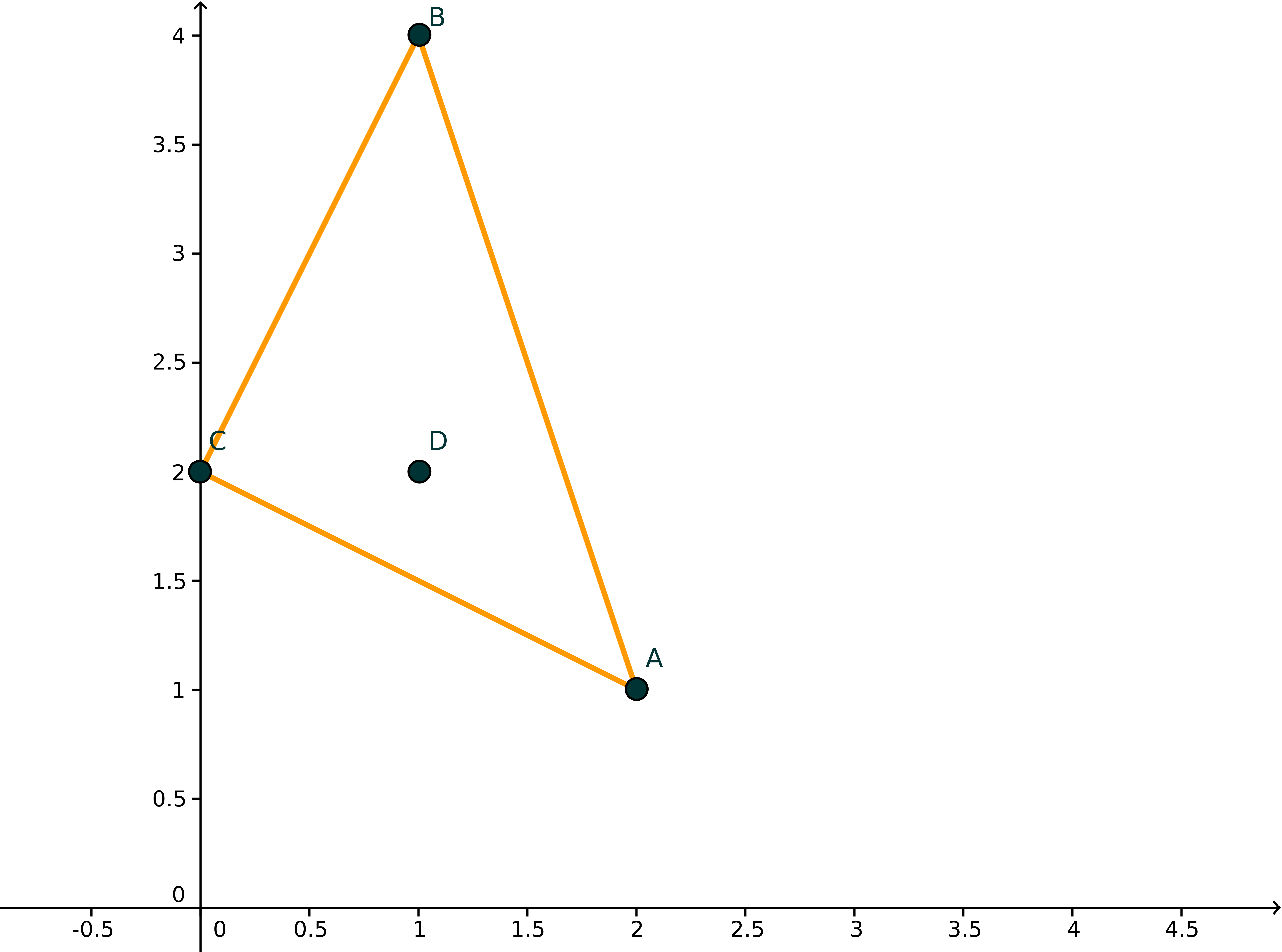}
\end{ex}

We recall the following theorem from \cite{MR1816701}.

\begin{thmGao} \label{gaoooo} Let $\corpo$ be a field and $f = ax^{m}+by^{n}+\sum_{i,j}c_{ij}x^{i}y^{j} \in \corpo[x,y]$ with $a,b \in \corpo^{*}$ and $(i,j)$ different from $(m,0), (0,n)$. Suppose that the Newton polytope of $f$ is contained in the triangle with vertices $(m,0),(0,n)$ and $(u,v)$ for some point $(u,v) \in \mathbb{R}^2$. If $gcd(m,n)=1$ then $f$ is absolutely irreducible over $\corpo$.
\end{thmGao}

\begin{proof} see \cite[Corollary 4.12]{MR1816701}.
\end{proof}

As a particular case, we have the following corollary.

\begin{cor} \label{2divisor} Let $\corpo$ be any field, $d \in \mathbb{Z}_{>1}$  and consider the following polynomial of degree $d$: $f(x,y) = x^{d-1}+y^{2d+1}+x^{d}y^{d}+x^{2d+1}y^{d-1}\in \corpo[x,y]$. Then $f$ is absolutely irreducible.
\end{cor}

\begin{proof} Let $A = (d-1,0)$, $B = (0,2d+1)$, $C = (d,d)$ and $D = (2d+1,d-1)$ be points in $\mathbb{R}^2$ and consider its convex hull: $\textbf{R}:= \langle A,B,C,D \rangle \subset \mathbb{R}^2$. Recall that by definition this is the set:
$
\mbox{\textbf{R}} = \{aA+bB+cC+dD\mid a+b+c+d = 1\}.
$
The convex hull of the points $A,B,D$ is the triangle with vertices $A,B,C$. Denote this triangle by $\textbf{T} := \langle A,B,D\rangle$.
We claim that $\textbf{R} = \textbf{T}$. Indeed, it is sufficient to show that $C \in \textbf{T}$. Note that,
$$
C = (d,d) = \frac{1}{3}(d-1,0)+\frac{1}{3}(0,2d+1)+\frac{1}{3}(2d+1,d-1) \in \textbf{T}.
$$

Now, consider the ideal generated $I = \langle d-1,2d+1\rangle \subset \mathbb{Z}$. Note that the condition
$d\not\equiv1 \mod 3$ implies that the ideal $I = \mathbb{Z}$. Indeed, we have:
$
I = \langle d-1,2d+1\rangle = \langle d-1,2d+1-(d-1) \rangle = \langle d-1,d+2\rangle = \langle d-1,3\rangle = \mathbb{Z}.
$
In particular, by the proposition of S.Gao we conclude that the polynomial $f$ is absolutely irreducible.
\end{proof}

\begin{prop} \label{2Jou} Let $\corpo$ be a algebraically closed field of characteristic two and $\mathcal{J}_d$ be the Jouanolou foliation of degree $d>1$ on $\Projdois$. If $d\equiv 0 \mod 2$ then $\mathcal{J}_d$ is $2$-closed and if $d\equiv 1 \mod 2$ then $\mathcal{J}_d$ is not $2$-closed with irreducible $2$-divisor.
\end{prop}

\begin{proof} The Jouanolou foliation of degree $d>1$ defined over a field $K$ is given on $D_{+}(z)\cong \mathbb{A}_{\corpo}^{2}$ by the following vector field
$
  v = (xy^{d}-1)\partial_x-(x^d-y^{d+1})\partial_y
$
and by computations we have:
$$
v^{2} = ((d+1)xy^{2d}-y^{d}-dx^{d+1}y^{d-1})\partial_x+((d+1)y^{2d+1}-(2d+1)x^dy^d+dx^{d-1})\partial_y
$$
and $D_{2}(v) = \Delta_{\mathcal{F}}|_{D_{+}(z)} =  v^{2}(y)v(x)-v^{2}(x)v(y)$ is given by
$$D_{2}(v) = (2d+2)xy^{3d+1}-(d+2)y^{2d+1}-(4d+2)x^{d+1}y^{2d}+(3d+2)x^{d}y^{d}+dx^{2d+1}y^{d-1}-dx^{d-1}.$$
In particular, if $K = \corpo$ we have that $\mathcal{J}_d$ is $2$-closed if $d\equiv 0 \mod 2$, since $D_2(v) = 0$ in that case. On the other hand, if $d\equiv 1\mod 2$ then $v\wedge v^2  = f(x,y)\partial_{x}\wedge \partial_{y}$ where $f(x,y)$ is the local equation to the $2$-divisor associated to $\mathcal{J}_d$ given by
$
f(x,y) = y^{2d+1}+x^{d}y^{d}+x^{2d+1}y^{d-1}+x^{d-1}
$
and by the Corollary \ref{2divisor} we conclude that the $2$-divisor is irreducible of  degree $3d$.
\end{proof}

\begin{lemma}\label{grauuuu} Let $\mathcal{J}_d$ be the Jouanolou foliation on the complex projective plane of degree $d\not\equiv 1\mod 3$. If $\mathcal{J}_d$ has a reduced $\mathcal{J}_d$-invariant algebraic curve then there is a $\mathcal{J}_d$-invariant reduced curve $D$ such that $\deg(D) = d+2$.
\end{lemma}
\begin{proof} Let $D$ the curve obtained by conjugation of $C$ via $\Aut(\mathcal{J}_d)$. This curve is reduced and it is invariant by $\Aut(\mathcal{J}_d)$. By the properties of the automorphism group of $\mathcal{J}_d$, it follows that the pair $(\mathcal{J}_d,D)$ is special (see \cite[Lemma 3]{MR2156709}). The argument now is very similar to the argument on Corollary \ref{grauspecial}, but with computations over the complex numbers (see also the argument after \cite[Lemma 2]{MR2156709}).
\end{proof}

Now we can present the main result of this section theorem.

\begin{thm} A very generic foliation of odd degree $d>1$ with $d\not\equiv 1 \mod 3$ on the complex projective plane does not have algebraic solutions.
\end{thm}

\begin{proof} It is sufficient to show that the Jouanolou foliation has no algebraic solution. Let $\mathcal{J}_d$ be the Jouanolou foliation with $d\equiv 1 \mod 2$ and $d\not\equiv 1\mod 3$. By the Proposition \ref{2Jou} we know that the foliation $\mathcal{F}_d := \mathcal{J}_d\otimes \mathbb{F}_2$ is not $2$-closed with $2$-divisor $\Delta_{\mathcal{F}_d}$ irreducible. Suppose by contradiction that there is $C = \{F = 0\}$ an irreducible invariant curve defined by a reduced polynomial $F \in \mathbb{C}[x,y,z]$. We can assume that $C$ is defined by
a irreducible polynomial over $\mathbb{Z}$. Indeed, let $K$ be the smallest field extension of $\mathbb{Q}$ where all coefficients
of $F$ are defined.
Clearing denominators we can assume
$F\in \mathcal{O}_{K}[x,y,z]$. Let $G_{K/\mathbb{Q}}$ the Galois group of $K/\mathbb{Q}$ and for $\sigma \in G_{K/\mathbb{Q}} $ denote by $F^{\sigma}$ the natural action of $\sigma$ on the coefficients of $F$. Consider the polynomial $H = \prod_{\sigma \in G_{K/\mathbb{Q}}}F^{\sigma}$. By this construction we conclude that $H \in\mathbb{Z}[x,y,z]$ and it is $\mathcal{J}_d$-invariant. In particular, some irreducible factor $Q$ of $H$ gives a $\mathcal{J}_d$-invariant curve defined over $\mathbb{Z}$. In particular, it is reduced over $\mathbb{C}$. By the Lemma \ref{grauuuu} we know that the conjugation of $Q$, via $\Aut(\mathcal{J}_d)$, defines a $\mathcal{J}_d$-invariant curve $\tilde{C} = \{\tilde{Q}=0\}$ such that $\deg(\tilde{Q}) = d+2\equiv 1 \mod 2$. Let $L$ the smallest field extension of $\mathbb{Q}$ such that all coefficients of $\tilde{Q}$ are defined. By clearing denominators we can assume that
$\tilde{Q}\in \mathcal{O}_{L}[x,y,z]$. Let $\mathfrak{m}\in \spm(\mathcal{O}_L)$ be a maximal ideal over the maximal ideal
$2\mathbb{Z}$, i.e. $2\mathbb{Z} = \mathfrak{m} \cap \mathbb{Z}$. Since $\tilde{Q}$ is not a $2$-factor and since reduction modulo $\mathfrak{m}$ preserves the invariance we can find $Q_2$ a irreducible fator of $\tilde{Q}\otimes \mathbb{F}_\mathfrak{m}$ which defines a $\mathcal{F}_d$-invariant curve. By the irreducibility of $\Delta_{\mathcal{F}_d}$ we conclude that $Q_2 = \Delta_{\mathcal{F}_d}$ and so $d+2 \geq \deg(Q_2) = \deg(\mathcal{F}_d) = 3d$, contradiction since $d>1$.
\end{proof}

\normalfont \noindent \textbf{Acknowledgements.} W.Mendson thanks J.V. Pereira for conversations about some topics related to Section \ref{pdivisorP1P1}. The author acknowledges support from CAPES/COFECUB.

\nocite{MR1440180}\nocite{MR3687427}\nocite{galindo2020foliations}\nocite{MR3328860}\nocite{mendson2022}

\nocite{MR1440180}\nocite{MR3687427}\nocite{galindo2020foliations}\nocite{MR3328860}

\bibliographystyle{siam}

\bibliography{annot}

\end{document}